\documentclass[reqno,11pt,letterpaper,amsmath,amssymb]{amsart}
\usepackage{amsmath,amssymb,amsthm,graphicx,mathrsfs,url}
\usepackage[applemac]{inputenc}
\usepackage[usenames,dvipsnames]{color}
\usepackage[colorlinks=true,linkcolor=Red,citecolor=Green]{hyperref}
\usepackage{enumitem}
\usepackage[russian, francais, english]{babel}
\usepackage{csquotes}
\usepackage[all]{xy}

\usepackage[margin=1.4in]{geometry}

\newtheorem{theo}{Theorem}
\newtheorem{prop}{Proposition}[section]

\newtheorem{lemm}[prop]{Lemma}
\theoremstyle{definition}
\newtheorem{corr}[theo]{Corollary}

\numberwithin{equation}{section}

\newcommand{\R}{\mathbb{R}}
\newcommand{\N}{\mathbb{N}}
\renewcommand{\C}{\mathbb{C}}
\newcommand{\Sc}{\mathcal{S}}
\newcommand{\Qc}{\mathcal{Q}}
\newcommand{\Tc}{\mathcal{T}}

\newcommand{\Pc}{\mathcal{P}}

\newcommand{\Z}{\mathbb{Z}}

\newcommand{\dd}{\mathrm{d}}
\newcommand{\Lie}{\mathcal{L}}

\newcommand{\e}{\mathrm{e}}

\newcommand{\strf}{{\tr}_\mathrm{s}^\flat}

\newcommand{\dom}{\mathcal{O}}

\newcommand{\wl}{\mathrm{wl}}

\let\Re=\Real

\DeclareMathOperator{\supp}{supp}

\DeclareMathOperator{\WF}{WF}

\DeclareMathOperator{\tr}{tr}

\title{Closed billiards trajectories with prescribed bounces}
\author{Yann Chaubet}
\begin{document}

\maketitle

\begin{abstract}
We give the asymptotic growth of the number of primitive periodic trajectories of a two dimensional dispersive billiard, when we prescribe their number of bounces on one of the obstacles.
\end{abstract}

\section*{Introduction}\label{sec:intro}
Consider $D_0, D_1, \dots, D_r \subset \R^2$ ($r \geqslant 3$) some compact and strictly convex open sets, with smooth boundaries $\partial D_0, \dots, \partial D_r.$ We assume that $D_i \cap D_j = \emptyset$ whenever $i \neq j$. We moreover assume that the billiard $\bold{B'} = \{D_0, D_1, \dots, D_r\}$ satisfies the non-eclipse condition, that is,
$$
\mathrm{conv}(D_i \cup D_j) \cap D_k = \emptyset, \quad k \neq i,j,
$$
where $\mathrm{conv}(A)$ denotes the convex hull of a set $A$. We will denote $D = \bigcup_j D_j$. A \textit{billiard trajectory} is a piecewise Euclidian trajectory\footnote{By "Euclidian" we mean trajectories going in a straight line with constant speed 1.} $\gamma : I \to \R^2 \setminus D^\circ$ (here $I \subset \R$ is an interval) which rebounds on each $\partial D_j$ according to Fresnel Descartes' law (see Figure \ref{fig:billiard_trajectory}). 
\begin{figure}[h]
\includegraphics[scale=0.6]{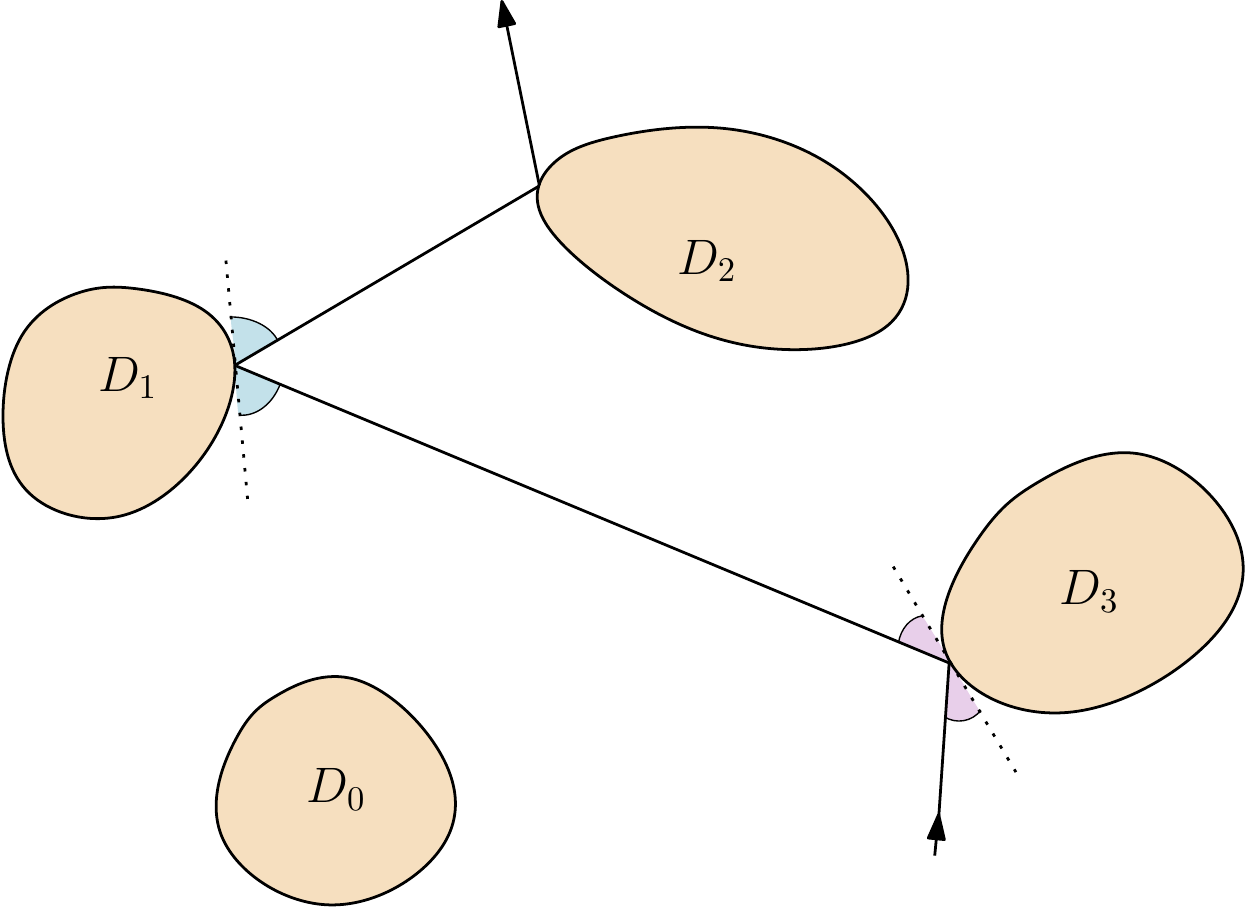}
\caption{A billiard trajectory}
\label{fig:billiard_trajectory}
\end{figure}
A trajectory $\gamma : [0, \tau] \to \R^2 \setminus D^\circ$ will be said to be \textit{closed} if $\gamma(0) = \gamma(\tau)$ and  $\gamma'(0) = \gamma'(\tau)$;  a closed trajectory will be said to be \textit{primitive} if $\gamma|_{[0, \tau']}$ is not closed for every $\tau'< \tau.$ We will identify two closed trajectories $\gamma_j : \R/\tau_j \Z \to \R^2 \setminus D^\circ$ ($j=1,2 $) whenever $\tau_1 = \tau_2$ and $\gamma_1(\cdot) = \gamma_2(\cdot + \tau)$ for some $\tau \in \R.$ Denote by $\Pc_\bold{B'}$ the set of primitive closed trajectories of the billiard table $\bold{B'}$. Then a result of Morita \cite{morita1991symbolic} states that there is $h_\bold B' > 0$ such that
\begin{equation}\label{eq:pot}
\sharp\{\gamma \in \Pc_\bold{B'}~:~\tau(\gamma) \leqslant t\} \sim \frac{\e^{h_\bold{B'}t}}{h_\bold{B'} t}, \quad t \to \infty,
\end{equation}
where $\tau(\gamma)$ denotes the period of a periodic trajectory $\gamma.$

The purpose of the present paper is to give the asymptotic growth of the number of primitive closed trajectories of $\bold{B'}$ when we additionnaly prescribe their number of rebounds on $D_0.$ More precisely, for $\gamma \in \Pc_\bold{B'}$ we denote by $r(\gamma)$ the number of rebounds of $\gamma$ on $D_0$; we have the following result.

\begin{theo}\label{thm:main}
There are $c, h_\bold{B} > 0$ such that for every $n \geqslant 1$, it holds
\begin{equation}\label{eq:main}
\sharp\{\gamma \in \Pc_\bold{B'}~:~\tau(\gamma) \leqslant t,~r(\gamma) = n\} \sim \frac{(ct)^n}{n!} \frac{\e^{h_\bold{B}t}}{h_\bold{B} t}, \quad t \to \infty.
\end{equation}
Moreover $h_\bold{B}$ depends only on the billiard table $\bold{B} = \{D_1, \dots, D_r\}.$
\end{theo}

As we will see in \S\ref{sec:apriori}, by using the symbolic representation of the billiard flow and (\ref{eq:pot}), one can prove that for some constants $a, b > 0$ we have
$$
at^{n-1}\exp(h_\bold{B}t) \leqslant \sharp\{\gamma \in \Pc_\bold{B'}~:~\tau(\gamma) \leqslant t,~r(\gamma) = n\} \leqslant bt^{n-1}\exp(h_\bold{B}t)
$$
provided $t$ is large enough; yet this method do not \textit{a priori} provide the more precise asymptotics (\ref{eq:main}).

Our approach for proving (\ref{eq:main}) is reminiscent of a previous work \cite{chaubet2021closed} about the asymptotic growth of the number of closed geodesics on negatively curved surfaces for which certain intersection numbers are prescribed. In particular we make use of the work of Dyatlov--Guillarmou \cite{dyatlov2016pollicott} about the existence of Pollicott--Ruelle resonances for open hyperbolic systems (the recent work of K\"uster--Sch\"utte--Weich \cite{kuster2021smooth} details how a hyperbolic billiard flow can be described by the framework of \cite{dyatlov2016pollicott}). This allows to obtain a microlocal description of the transfer operator $\Tc(s)$ associated to the first return map (of the billiard flow) to $\pi^{-1}(\partial D_0)$ (here $\pi : S\R^2 \to \R^2$ is the natural projection), weighted by $\exp(-s t_0(\cdot))$ where $t_0(\cdot)$ is the first return time to $\pi^{-1}(\partial D_0)$ (see \S\ref{sec:adding}), and to apply a Tauberian theorem of Delange to the (transversal) trace of the composition\footnote{Actually, we compute the trace of $(\varrho \Tc(s))^n$ for some cutoff function $\varrho \in C^\infty(\pi^{-1}(\partial D_0), [0,1]).$} $\Tc(s)^n$ (which is linked to some dynamical zeta function involving the periodic orbits rebounding $n$ times on $\partial D_0$).

Similar asymptotics for open dispersive billiards in $\R^d$ ($d \geqslant 3$) could also be obtained with our methods; however here we restrict ourselves to the case $d = 2$  for the sake of simplicity.

\subsection*{Related works}
In \cite{morita1991symbolic} Morita proves the asymptotics (\ref{eq:pot}) by constructing a symbolic coding of the billiard flow and by using the work of Parry-Pollicott \cite{parry1983analogue}. Later, Stoyanov \cite{stoyanov2012non} proved the more precise asymptotics
$$
\sharp\{\gamma \in \Pc_\bold{B'}~:~\tau(\gamma) \leqslant t\} = \int_2^x \frac{\dd t}{\log t} + O(\e^{ct}), \quad t \to +\infty,
$$
for some $c \in \left]0,  h_\bold B\right[$, by proving some non-integrability condition over the non-wandering set and by using Dolgopyat-type estimates (see also \cite{petkov2012distribution} for an asymptotics of the number of primitive closed trajectories with periods lying in exponentially shrinking intervals. We finally mention the book of Pektov--Stoyanov \cite{pektov1992geometry}.

\subsection*{Organization of the paper}
The paper is organized as follows. In \S\ref{sec:preliminaries} we present some geometrical and dynamical tools. In \S\ref{sec:adding} we introduce the weighted transfer operator associated to the first return map to $\partial D_0$ and we compute its Attiyah-Bott transversal trace. In \S\ref{sec:tauberian} we make use of a Tauberian argument. In \S\ref{sec:apriori} we prove some \textit{a priori} estimates on $\sharp\{\gamma \in \Pc_\bold{B'}~:~\tau(\gamma) \leqslant t,~r(\gamma) = n\}$. Finally in \S\ref{sec:proof} we combine the results of \S\S\ref{sec:tauberian},\ref{sec:apriori} to prove Theorem \ref{thm:main}.

\subsection*{Acknowledgements}
I thank Colin Guillarmou for fruitful discussions and for his relecture of the present work, as well as Fr\'ed\'eric Naud for suggesting to consider this problem, which is somehow analogous to the one considered in \cite{chaubet2021closed}. Finally I thank Benjamin K\"uster, Philipp Sch\"utte and Tobias Weich for important discussions about their recent work \cite{kuster2021smooth}. This project has received funding from the European Research Council (ERC) under the European Unions Horizon 2020 research and innovation programme (grant agreement No. 725967).

\section{Preliminaries} \label{sec:preliminaries}
In this section we expose some well known facts about open dispersive billiards.

\subsection{The billiard flow}
Let $D_1, \dots, D_r \subset \R^2$ be pairwise disjoint compact convex obstacles, where $r \in \N_{\geqslant 3}.$ We denote by $S\R^2$ the unit tangent bundle of $\R^2$ and $\pi : S\R^2 \to \R^2$ the natural projection. For $x \in \partial D_j$, we denote by $n_j(x)$ the outward unit normal vector to $\partial D_j$ at the point $x$. We define the (non glancing) billiard table $M$ as 
$$
M = N / \sim, \quad N =  S\R^2 \setminus \left(\pi^{-1}(D^\circ) \cup G\right),
$$
where $G = T\partial D$ and $D = \bigcup_{j=1}^r D_j$, and where $(x,v) \sim (y,w)$ if and only if
$$
x = y \in \partial D_j, \quad w = v - 2\langle v, n_j(x) \rangle n_j(x), \quad j = 1,\dots, r.
$$
The set $M$ is endowed with the quotient topology. We denote by $\varphi = (\varphi_t)_{t \in \R}$ the billiard flow (which is defined on an open subset of $\R \times M$). The manifold $M$ can be endowed with a differential structure by declaring that flow charts near $\pi^{-1}(\partial D)$ are smooth charts for $M$; we refer to the work of K\"uster--Sch\"utte--Weich \cite{kuster2021smooth} for a detailed exposition of the construction of this differential structure.  The flow $\varphi$ becomes a smooth flow on $M$ and we denote by $X$ the associated vector field. Define the Liouville one form $\alpha \in \Omega^1(\R^2 \times S^1)$ by
\begin{equation}\label{eq:defliouville}
\langle \alpha(x,v), \eta \rangle =  \langle \dd \pi(x,v) \cdot \eta, v \rangle, \quad (x,v) \in \R^2 \times S^1, \quad \eta \in T_{(x,v)} (\R^2 \times S^1),
\end{equation}
where $\pi : \R^2 \times S^1 \to \R^2$ is the projection over the first factor. 

\begin{lemm}
The form $\alpha$ gives rise to a one-form on $M$ which is smooth outside the glancing set $G$. Moreover $\alpha$ is of contact type and $X$ is its associated Reeb vector field, that is $\alpha \wedge \dd \alpha$ is a volume form and
$$
\iota_X \alpha = 1, \quad \iota_X \dd \alpha = 0,
$$
where $\iota_X$ denotes the interior product.
\end{lemm}

\begin{proof}
The assertion is clear far from $\pi^{-1}(D)$. Thus we take $j \in \{1, \dots, r\}$ and some $(x, v) \in \pi^{-1}(\partial D_j) \setminus G_j$ (here $G_j = T \partial D_j$). Let $\psi :~]-\varepsilon, \varepsilon[ \to \R^2$ be a parametrization of $\partial D_j$ near $x$ such that $|\psi'(s)| = 1$ for $s \in ]-\varepsilon, \varepsilon[$ and $\psi(0) = x$. Let 
$\varphi_0$ be the (counterclockwise) angle between $n(x)$ and $v$. Let $U$ be a small neighborhood of $(0, \varphi_0, 0)$ in $\R^3$, and set 
$$V = U \sqcup U', \quad U' = \{(s, \pi - \varphi, \tau)~:~(s, \varphi, \tau) \in U\}.$$
As $\varphi_0 \notin \pi/2 + \pi \Z$, we may assume (up to shrinking $U$) that the map $\Psi : V \to \R^2 \times S^1$ defined by
$$
\Psi ~ : ~(s, \varphi, \tau) \mapsto \Bigl(\psi(s) + \tau R_\varphi n(s),~R_\varphi n(s) \Bigr)
$$
is a smooth local embedding, where $R_\varphi : S^1 \to S^1$ is the rotation of angle $\varphi$ and where we set $n(s) = n(\psi(s))$. From (\ref{eq:defliouville}), it is then immediate to check that, in the coordinates $(s, \varphi, \tau)$, we have
\begin{equation}\label{eq:alphacoordinates}
\alpha = \sin(\varphi) \dd s + \dd \tau.
\end{equation}
In particular we have $Q^* \alpha = \alpha$ where $Q$ denotes the map $(s, \varphi, \tau) \mapsto (s, \pi - \varphi, \tau)$. Let $p : N \to M$ be the natural projection. We have a map $\Phi : U \to M$ defined by 
$\Phi(z) = p(z)$ if $\Psi(z) \in N$ and $\Phi(z) = p(Q(z))$ otherwise. Up to shrinking $U$, this map realizes a local diffeomorphism from $U$ to a neighborhood of $p(x,v)$ in $M$ (this map is a chart for $M$ near $p(x,v)$, see \cite[\S4]{kuster2021smooth}). Since $Q^*\alpha = \alpha$, it follows that $\alpha$ gives rise to a smooth one-form on $\Phi(U).$ The fact that $\alpha$ is a contact form follows from the expression (\ref{eq:alphacoordinates}). In the coordinates $(s, \varphi, \tau)$, $X$ is represented by $\partial_\tau$; thus $X$ satisfies the announced Reeb conditions. \end{proof}

\subsection{The Anosov property}\label{subsec:anosov}

The trapped set $\Lambda$ is defined as the set of points of $z \in M$ which satisfy
$$\sup T(z) = - \inf T(z) = +\infty \quad \text{where} \quad T(z) = \{t \in \R~:~\pi(\varphi_t(z)) \in \partial D\}.$$
We define the first (future and the past) return times to $\partial D$ by
$$
t_\pm(z) = \inf\{t > 0~:~\pi(\varphi_{\pm t}(z)) \in \partial D\}, \quad z \in M.
$$
Let $K = \Lambda\cap \partial D$. The (future and past) billiard maps $B_\pm : K \to K$ are then defined as
$$
B_\pm(z) = \varphi_{\pm t_\pm(z)}(z), \quad z \in K.
$$
By \cite{morita1991symbolic} (see also \cite[\S4.4]{chernov2006chaotic}), the billiard flow is uniformly hyperbolic, meaning that for each $z \in \Lambda$ there is $\dd \varphi_t-$invariant decomposition
$$
T_z\Lambda= \R X(z) \oplus E_u(z) \oplus E_s(z),
$$
which depends continuously on $z$, and such that for some $C, \nu > 0$ independent of $z \in \Lambda$, we have for some smooth norm $|\cdot|$ on $TM$,
$$
\left|\dd \varphi_t(z) v\right| \leqslant \left \lbrace \begin{matrix} C \e^{-\nu t} |v|,~ &v \in E_s(z),~&t\geqslant 0, \vspace{0.2cm}  \\  C \e^{-\nu |t|} |v|,~ &v \in E_u(z),~&t\leqslant 0. \end{matrix} \right.
$$

\subsection{The non-eclipse condition}
We will assume that our billiard table is non-eclipsing, in the sense that for all $1 \leqslant i < j < k \leqslant r$ we have
$$
\mathrm{Conv}(D_i \cup D_j) \cap D_k = \emptyset,
$$
where $\mathrm{Conv}(A)$ denotes the convex hull of a set $A$. By \cite{morita1991symbolic}, $B_\pm : K \to K$ is H\"older conjugated to a subshift of finite type. More precisely, let
$
\mathcal{A} = \{1, \dots, r\}
$
and 
$$
\Sigma = \{(u_n) \in \mathcal{A}^\Z~:~u_n \neq u_{n+1},~n \in \Z\}.
$$
Let $\sigma_\pm : \Sigma \to \Sigma$ be the map $(u_n) \mapsto (u_{n\pm 1})$. We endow $\Sigma$ with the topology coming from the distance
$$
\dd_\Sigma(u, v) = \sum_{n \in \Z} 2^{-|n|} |u_n - v_n|.
$$
Then there is a homeomorphism $\psi : K \to \Sigma$, which is H\"older continuous, such that 
$$
\sigma_\pm \circ \psi = \psi \circ B_\pm.
$$
In fact, $\psi$ is simply given by
$$
\psi(z)_n = j, \quad B_+^n(z) \in \partial D_j, \quad z \in K, \quad n \in \Z.
$$
Of course the billiard flow $(\varphi_t)$ is conjugated to the suspension of $\sigma_\pm$ associated to the time return map $t_\pm \circ \psi^{-1}$. This means that we have a H\"older homeomorphism
$$
\Psi : \Lambda \to (K \times \R_+) / \sim
$$
where 
$
(z, t_+(z)) \sim (B_+(z), 0)
$
for $z \in K.$ In the coordinates $(z, t)$, $X$ is simply represented by $\partial_t$. In what follows, we will denote by $\psi_n(z)$ the $n$-th term of the sequence $\psi(z).$ An immediate consequence of the existence of a conjugacy $\Psi$ as above is the following

\begin{lemm}\label{lem:exponentiallyclose}
There is $C > 0$ and $\beta > 1$ such that the following holds. Assume that $z, z' \in K$ satisfy
$$
\psi_n(z) = \psi_n(z'), \quad |n|\leqslant N.
$$
Then 
$
\dd(z, z') \leqslant C \beta^{-N}
$.
\end{lemm}

\subsection{Isolating blocks}
In this subsection we show that we can work with the framework of \cite{dyatlov2016pollicott} (see also \cite[\S5]{kuster2021smooth} for a more detailed exposition). We have 
$$
\Lambda = \bigcap_{t \in \R} \varphi_t(V \setminus TD)
$$
where $V = \{z \in M~:~T_-(z) \neq \emptyset \text{ and }T_+(z) \neq \emptyset \} \subset M$. Here we set
$$
T_\pm(z) = \{t \in T(z)~:~\pm t > 0\}.
$$
By \cite[Theorem 1.5]{conley1971isolated}, $\Lambda$ is the maximal invariant set in some isolating block. More precisely, there exists a relatively compact neighborhood $U \subset M$ of $K$ such that $\partial U$ is smooth and 
$$\partial_0 U = \{(x,v) \in \partial U~:~v \in T_x \partial U\}$$
is a smooth submanifold of $\partial U$ of codimension $1$, and with the property that for some $\varepsilon > 0$ one has 
$$
z \in \partial_0 U \quad \implies \quad \forall |t| \in \left]0, \varepsilon \right[ ,~\varphi_t(z) \notin U.
$$
By proceeding as in \cite[Lemma 2.3]{guillarmou2017boundary}, we may find a vector field $\widetilde X$ on $M \setminus TD$ such that $X-\widetilde X$ is supported in an arbitrary small neighborhood of $\partial_0U$, which is arbitrarily small in the $C^\infty$ topology and such that for any boundary defining function $\rho : U \to \R_{\geqslant 0}$ of $\partial U$\footnote{This means that $\rho > 0$ on $U$, $\rho = 0$ on $\partial U$ and $\dd \rho \neq 0$ on $\partial U$.}, we have, for any $z \in \partial U$,
$$
\widetilde X \rho(z) = 0 \quad \implies \quad \widetilde X^2\rho(z) < 0.
$$
Moreover, we have $\Gamma_\pm(U) = \widetilde \Gamma_\pm(U)$ where 
$$
\Gamma_\pm(U) = \{z \in U~:~\varphi_t(z) \in U,~ \mp t \geqslant 0\}, \quad \widetilde \Gamma_\pm(U) = \{z \in U~:~\widetilde \varphi_t(z) \in U,~ \mp t \geqslant 0\},
$$
where $\widetilde \varphi_t$ denotes the flow of $\widetilde X$. Note also that $\mathrm{dist}(\Gamma_\pm(U), \partial_0 U)>0.$ For simplicity, we will denote $\Gamma_\pm = \Gamma_\pm(U).$

By \cite[Lemma 2.10]{dyatlov2016pollicott}, there are two vector subbundles $E_{\pm} \subset T_{\Gamma_\pm} U$ with the following properties :
\begin{enumerate}
\item $E_+|_K = E_u$, $E_-|_K = E_s$ and $E_\pm(z)$ depends continuously on $z \in \Gamma_\pm$;

\item For some constants $C',\nu' > 0$ we have
$$
\|\dd \widetilde \varphi_{\mp t} (z)v\| \leqslant C' \e^{-\nu' t}\|v\|, \quad v \in E_\pm(z), \quad z \in \Gamma_\pm, \quad t \geqslant 0;
$$
\item If $z \in \Gamma_\pm$ and $v \in T_zU$ satisfy $\langle \alpha(z), v \rangle = 0$ and $v \notin E_\pm(z)$, then as $t \to \mp \infty$
$$
\left\|\dd \widetilde \varphi_{t}(z)v \right\| \rightarrow \infty,\quad \frac{\dd \widetilde \varphi_t(z)v }{\left\|\dd \widetilde \varphi_t(z)v\right\|} \rightarrow E_{\mp}|_K.
$$
\end{enumerate}

\subsection{The resolvent of the billiard flow}\label{subsec:resolvent}

For $\Re(s) \gg 1$, we define the (future and past) resolvents $\widetilde R_\pm(s) : \Omega^\bullet_c(U) \to \mathcal{D}'^\bullet(U)$ by
$$
\widetilde R_{\pm}(s)\omega(z) =  \pm \int_{0}^{\widetilde t_{\mp,U}(z)} \widetilde \varphi_{\mp t}^*\omega(z) \e^{-ts} \dd t, \quad \omega \in \Omega^\bullet_c(U), \quad z \in U,
$$
where we set
$$
\widetilde t_{\pm, U}(z) = \inf\{t > 0~:~\widetilde \varphi_{\pm t}(z) \in \partial U\}, \quad z  \in U.
$$
Here $\Omega^\bullet_c(U)$ denotes the space of smooth differential forms which are compactly supported in $U$ while $\mathcal{D}'^\bullet(U)$ denotes the space of currents in $U$ (that is $\mathcal{D}'^k(U)$ is the dual space of $\Omega_c^{3-k}(U)$ for $k = 0, \dots, 3$).
Note that 
$$
\left(\Lie_{\widetilde X} \pm s\right) \widetilde R_\pm(s) = \widetilde R_\pm(s) \left(\Lie_{\widetilde X} \pm s \right) = \mathrm{Id}_{\Omega^\bullet_c(U)}.
$$
Then by \cite{dyatlov2016pollicott}, the family $s \mapsto \widetilde R_{\pm}(s)$ extends to a family of operators meromorphic in the parameter $s \in \C$, whose poles have residues of finite rank. Denote by $\mathrm{Res}(\widetilde X)$ the set of those poles. Near any $s_0 \in \mathrm{Res}(\widetilde X)$ we have for some finite rank projector $\Pi_\pm(s_0) : \Omega^\bullet \to \mathcal{D}'^\bullet$
$$
R_\pm(s) = H_\pm(s) + \sum_{j = 1}^{J(s_0)} \frac{(X \pm s)^{j - 1} \Pi_{\pm}(s_0)}{(s - s_0)^j}
$$
where $s \mapsto H_{\pm}(s)$ is holomorphic near $s_0$. Moreover we have $\supp(\Pi_\pm(s_0)) \subset \Gamma_\pm \times \Gamma_\mp$ and
\begin{equation}\label{eq:wfset}
\WF'(H_\pm(s)) \subset \Delta(T^*U) \cup \Upsilon_\pm \cup (E_\pm^* \times E_\mp^*), \quad \WF'(\Pi_\pm(s_0)) \subset E_\pm^* \times E_\mp^*,
\end{equation}
where $\Delta(T^*U) = \{(\xi, \xi),~\xi \in T^*U \} \subset T^*(U \times U)$ and 
$$
\Upsilon_\pm = \{(\Phi_t(z, \xi), (z,\xi)),~\pm t \geqslant 0,~\langle \xi, X(z) \rangle = 0,~z \in U~\widetilde \varphi_t(z) \in U\}.
$$
Here $\Phi_t$ denotes the symplectic lift of $\varphi_t$ on $T^*U$, that is 
$$
\Phi_t(z, \xi) = (\varphi_t(z), (\dd_z\varphi_t)^{-\top} \xi), \quad (z, \xi) \in T^*U, \quad \varphi_t(z) \in U,
$$ 
and the subbundles $E_\pm^* \subset T^*_{\Gamma_\pm}U$ are defined by 
$E_\pm^*(\R X(z) \oplus E_\pm) = 0.$ Also we denoted 
$$
\WF'(\widetilde R_\pm(s)) = \{(z, \xi, z', \xi') \in T^*(U \times U),~(z, \xi, z', -\xi') \in \WF(\widetilde R_\pm(s))\},
$$
where $\WF(\widetilde R_\pm(s)) \subset T^*(U \times U)$ is the H\"ormander wavefront set of (the Schwartz kernel of) $\widetilde R_\pm(s)$, see \cite[\S8]{hor1}, and 
$$
\WF'(\widetilde R_\pm(s)) = \{(z, \xi, z', \xi')~:~ (z, \xi, z, -\xi') \in \WF(\widetilde R_\pm(s))\}.
$$

\subsection{The scattering operator}\label{subsec:scatteringop}
We define
$$\widetilde \partial_\pm = \{z \in \partial U~:~\mp \widetilde X\rho(z) > 0\} \text{ and }\widetilde \partial_0 = \{z \in \partial U~:~ \widetilde X \rho(z) = 0\}.$$
The scattering map $\widetilde S_\pm : \widetilde \partial_\mp \setminus \Gamma_\mp \to \widetilde \partial_\pm \setminus \Gamma_\pm$ is defined by 
$$
\widetilde S_\pm(z) = \widetilde \varphi_{\widetilde t_{\pm, U}(z)}(z), \quad z \in \widetilde \partial_\mp \setminus \Gamma_\mp
$$ 
(see Figure \ref{fig:notations}).
The Scattering operator $\widetilde \Sc_\pm(s) : \Omega^\bullet_c(\widetilde \partial_\mp \setminus \Gamma_\mp) \to \Omega_c^\bullet(\widetilde \partial_\pm \setminus \Gamma_\pm)$ is then defined by
$$
\widetilde{\mathcal{S}}_{\pm}(s) \omega=\left(\widetilde{S}_{\mp}^{*} \omega\right) \mathrm{e}^{-s \widetilde t_{\mp, U}(\cdot)}, \quad \omega \in \Omega_{c}^{\bullet}\left(\partial_{\mp} \backslash \Gamma_{\mp}\right).
$$
Note that for $\Re(s) \gg 1$, $\widetilde \Sc_\pm(s)$ extends as an operator $C_c(\widetilde \partial_\mp, \wedge^\bullet T^*\widetilde \partial_\mp) \to C_c(\widetilde \partial_\pm, \wedge^\bullet T^*\widetilde \partial_\pm)$,  where $C_c(\widetilde \partial_\pm, \wedge^\bullet T^*\widetilde \partial_\pm)$ is the space of compactly supported continuous forms on $\widetilde \partial_\pm$, since for any $w \in \Omega^\bullet(U)$ and $t \in \R$ we have
$$
\|\varphi_t^*w\|_\infty \leqslant C\e^{C|t|} \|w\|_\infty.
$$
In what follows we let $\iota_\pm : \widetilde \partial_\pm \to U$ be the inclusion and $(\iota_{\pm})_* : \Omega_c^\bullet(\widetilde \partial_\pm) \to \mathcal{D}'^{\bullet + 1}(U)$ be the pushforward operator, which is defined by
$$
\int_{U} (\iota_{\pm})_*u \wedge v = \int_{\widetilde \partial_\pm} u \wedge \iota_\pm^*v, \quad u \in \Omega^\bullet_c(\widetilde \partial_\pm), \quad v \in  \Omega^\bullet(U).
$$

\begin{prop}\label{prop:scatresolv}
We have
\begin{equation}\label{eq:wf}
\WF(\widetilde R_\pm(s)) \cap N^*(\widetilde \partial_\pm \times \widetilde \partial_\mp) = \emptyset.
\end{equation}
In particular by \cite[Theorem 8.2.4]{hor1}, the operator $(\iota_\pm)^* \iota_X \widetilde R_\pm(s) (\iota_\mp)_*$ is well defined. Moreover, for $\Re(s) \gg 1$ large enough, we have
\begin{equation}\label{eq:scatresolv}
\widetilde \Sc_\pm(s) = (-1)^N (\iota_\pm)^* \iota_{\widetilde X} \widetilde R_\pm(s) (\iota_\mp)_* : \Omega^\bullet_c(\widetilde \partial_\mp) \to \mathcal{D}'^\bullet(\widetilde \partial_\pm),
\end{equation}
where $N : \mathcal{D}'^\bullet \to \mathcal{D}'^\bullet$ is the number operator\footnote{That is, $N(\omega) = k \omega$ for $\omega \in \Omega^k_c(\widetilde \partial_\pm).$}.
\end{prop}

\begin{proof}
By definition $\widetilde X(z)$ is transverse to $T_z \widetilde \partial_\pm$ for $z \in \widetilde \partial_\pm.$ In particular, if $(z, \xi) \in T^*\widetilde \partial_\pm$ satisfies $\langle \xi, \widetilde X(z) \rangle = 0$ and $\langle \xi, T_z\widetilde \partial_\pm \rangle = 0$ then $\xi = 0.$ As $\widetilde \partial_+ \cap \widetilde \partial_- = \emptyset$ we obtain (\ref{eq:wf}) by (\ref{eq:wfset}).

Now let $W_\pm \subset \widetilde \partial_\pm$ be open sets such that $\overline W_\pm \subset \widetilde \partial_\pm.$ As $\overline W_\pm \cap \partial_0 = \emptyset$, there is $\varepsilon > 0$ such that 
$
\tilde t_{\mp, U}(z) > \varepsilon
$
for every
$
z \in W_\pm.
$
In particular, the proof of \cite[Lemma 3.3]{chaubet2021closed} applies and leads to the fact that (\ref{eq:scatresolv}) holds when $\widetilde \Sc_\pm(s)$ is seen as an operator $\Omega_c^\bullet(\widetilde \partial_\mp \setminus \Gamma_\mp) \to \mathcal{D}'^\bullet(\widetilde \partial_\pm \setminus \Gamma_\pm)$. By \cite[Theorem 5.6]{bowen1975ergodic}, as $\Lambda$ is not an attractor, we have $\mu(\Gamma_\pm) = 0$ where $\mu$ is the measure $|\alpha \wedge \dd \alpha|$. Take $U_\pm \subset \widetilde \partial_\pm$ a small neighborhood of $\Gamma_\pm$ in $\widetilde \partial_\pm$ and $\delta > 0$ small enough. Since $\Gamma_\pm \cap \widetilde \partial_0 = \emptyset$, we may assume that the map
$$
U_\pm \times [0, \delta) \to U, ~(y, t) \mapsto \varphi_{\mp t}(y)
$$
realizes a smooth diffeomorphism onto its image. In particular, because $\varphi_{\mp t}(\Gamma_\pm) \subset \Gamma_\pm$ for $t > 0$, we have $\mu_{\widetilde \partial_\pm}(\Gamma_\pm \cap \widetilde \partial_\pm) = 0$ where $\mu_{\widetilde \partial_\pm}$ corresponds to the measure $|\iota_\pm^* \dd \alpha|$. Thus we may proceed by similar arguments given in the proof of \cite[Proposition 3.2]{chaubet2021closed} we obtain that (\ref{eq:scatresolv}) holds when $\widetilde \Sc_\pm(s)$ is seen as an operator $\Omega_c^\bullet(\widetilde \partial_\mp) \to \mathcal{D}'^\bullet(\widetilde \partial_\pm)$.
\end{proof}

\section{Adding an obstacle}\label{sec:adding}

In this section we add an other obstacle $D_0$ and we will consider some weighted transfer operator associated to the first return map to $\pi^{-1}(\partial D_0)$; we will use the description of its microlocal structure to define and compute its flat trace.

\subsection{Notations}\label{subsec:notations}

We add another convex obstacle $D_0$, and we assume that the billiard table $(D_0, D_1, \dots, D_r)$ satisfies the non-eclipse condition. We define 
$$M', \Lambda', K', (\varphi_t'), T_\pm', t_\pm', B_\pm'$$ in the same way we defined $M, \Lambda, K, (\varphi_t), T_\pm, t_\pm, B_\pm$ (see \S\ref{sec:preliminaries}) by replacing the billiard table $\mathbf{B} = \{D_1, \dots, D_r\}$ by the billiard table $\mathbf{B}' = \{D_0, D_1, \dots, D_r\}.$ Let 
$$
P_\pm' : \Lambda'^{(\pm 1)} \to \pi^{-1}(\partial D'), \quad z \mapsto \varphi_{\pm t'_\pm(z)}(z),
$$
where 
$$\Lambda'^{(\pm 1)} = \{z \in M'~:~t'_\pm(z) < \infty\}.$$
Let $V_0 \subset \pi^{-1}(\partial D_0)$ be a relatively compact neighborhood of $K' \cap \pi^{-1}(\partial D_0)$ such that $V_0 \cap T\partial D_0 = \emptyset$, and set 
$$
V_\pm = \{z \in \partial U \cap \Lambda'^{(\pm 1)}~:~P'_\pm(z) \in V_0\}.
$$
Note that $U$ is a subset of $M$. However we may see $U$ as a subset of $M'$ since $U$ does not intersect $\pi^{-1}(D_0)$.
We also let $W_\pm$ be a neighborhood of $\Lambda_\pm \cap \partial U$ in $\partial U$ such that $W_\pm \cap \supp(\widetilde X - X) = \emptyset$ and we set $Y_\pm = W_\pm \cap V_\pm.$ We take $\phi_\pm \in C^\infty_c(V_\pm, [0,1])$ (resp. $\psi_\pm \in C^\infty_c(W_\pm), [0,1])$ such that $\phi_\pm \equiv 1$ near $(B'_\pm)^{-1}(K')$ (resp. $\psi_\pm \equiv 1$ near $\Lambda_\pm$) ; we define 
$$\chi_\pm = \phi_\pm \psi_\pm \in C^\infty_c(Y_\pm).$$ Note that $P'_\pm$ realizes a diffeomorphism $V_\pm \to P'_\pm(V_\pm) \subset \pi^{-1}(\partial D_0)$ which we denote by $Q_\pm$. We define
$
\Qc_\pm(s) : \mathcal{D}'^\bullet_c(Y_\pm) \to \mathcal{D}'^\bullet_c(Z_\pm),
$
where $Z_\pm = Q_\pm(Y_\pm)$, by
$$
\Qc_\pm(s)w = \e^{-s t_\pm'(\cdot)}\left(Q_\pm^{-1}\right)^*w, \quad w \in \Omega^\bullet_c(Y_\pm)
$$
(see Figure \ref{fig:notations}).
\begin{figure}[h]
\includegraphics[scale=0.6]{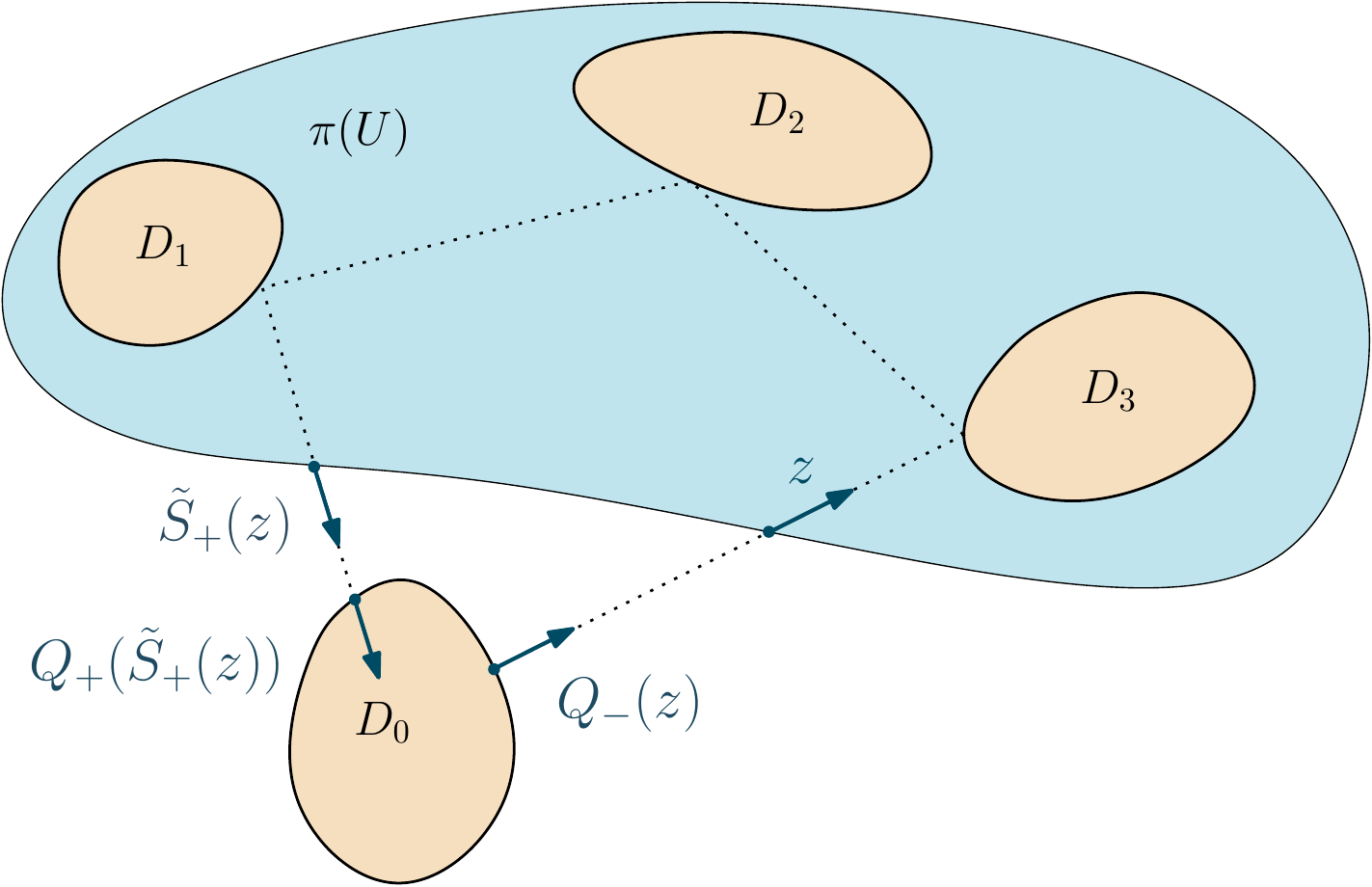}
\caption{The maps $Q_\pm$ and $\tilde S_\pm$}
\label{fig:notations}
\end{figure}
We finally set, with $Z_\pm = Q_\pm(Y_\pm) \subset \pi^{-1}(\partial D_0),$
\begin{equation}\label{eq:tcal}
\mathcal{T}_\pm(s) = \Qc_\pm(s) \chi_\pm \widetilde \Sc_\pm(s) \chi_\mp \Qc_\mp(s)^\top : \Omega^\bullet_c(Z_\mp) \to \mathcal{D}'^\bullet(Z_\pm). 
\end{equation}
The operator $\Tc_\pm(s)$ is the transfer operator associated to the first return map to $\pi^{-1}(\partial D_0)$ weigthed by $\e^{-s t_{0, \pm}(\cdot)}$, where $t_{0, \pm}(z) = \inf \{t > 0~:~ \pi( \varphi'_{\pm t}(z)) \in \partial D_0\}$ are the first (future are past) return times to $\partial D_0$ of a point $z \in \pi^{-1}(\partial D_0).$

\subsection{Composing the scattering maps}\label{subsec:composingscat}
Let $Z = Z_+ \cap Z_-$ and 
$$\varrho = (\chi_+ \circ Q_+^{-1})(\chi_- \circ Q_-^{-1}) \in C^\infty_c(Z).$$
We have the following result.

\begin{prop}\label{prop:composition}
For any $n \geqslant 2$, the composition $(\varrho \Tc_\pm(s))^n : \Omega^\bullet_c(Z) \to \mathcal{D}'^\bullet(Z)$ is well defined.
\end{prop}

\begin{proof}
By \cite[Theorem 8.2.4]{hor1} and Proposition \ref{prop:scatresolv} we have
\begin{equation}\label{eq:wfscat}
\WF'(\widetilde \Sc_\pm(s)) \subset \left\{\dd \left(\iota_\pm \times \iota_\mp\right)(z, z')^\top \cdot (\xi, \xi')~:~ (z, z', \xi, \xi') \in \WF'\left(\widetilde R_\pm(s)\right)\right\},
\end{equation}
where $\iota_\pm \times \iota_\mp : \widetilde \partial_\pm \times \widetilde \partial_\mp \hookrightarrow U \times U$ is the inclusion. To prove that $(\varrho \Tc_\pm(s))^2$ is well defined, it suffices to show by \cite[Theorem 8.2.14]{hor1} that 
$
A \cap B_1 = \emptyset
$
where
$$
A = \left\{(z, \xi) \in T^*Z~:~\exists z^{\prime} \in Z,\left(z^{\prime}, 0, z, \xi\right) \in \mathrm{WF}^{\prime}\left(\varrho {\mathcal{T}}_{\pm}(s)\right)\right\}
$$
and 
$$
B_1 = \left\{(z, \xi)\in T^*Z~:~ \exists z^{\prime} \in Z,\left(z, \xi, z^{\prime}, 0\right) \in \mathrm{WF}\left(\varrho {\mathcal{T}}_{\pm}(s)\right)\right\}.
$$
Note that $\left(\dd \iota_\pm(z)^\top\right)|_{\ker X(z)} : \ker X(z) \to T^*_z \widetilde \partial_\pm$  is injective for any $z \in \widetilde \partial_\pm$, since $X(z)$ is transverse to $T_z \widetilde \partial_\pm$. Moreover $Q_\pm : Y_\pm \to Z_\pm$ is a diffeomorphism and thus $\dd Q_\pm^{-1}(z)^\top : T^*_z \widetilde \partial_\pm \to T^*_z\pi^{-1}(\partial D_0)$ is injective for any $z \in Y_\pm.$ Now by (\ref{eq:tcal}) we have
$$
\WF'(\varrho \Tc_\pm(s)) \subset \dd (Q_\pm \times Q_\mp)^\top \left(\WF'(\widetilde \Sc_\pm(s)) \cap \supp(\chi_\pm \times \chi_\mp) \right),
$$
Moreover by (\ref{eq:wf}) and (\ref{eq:wfscat}) we have
$$
\WF'(\widetilde \Sc_\pm(s)) \subset \dd(\iota_\pm \times \iota_\mp)^\top\left(\Upsilon_\pm \cup (E_\pm^* \times E_\mp^*)\right),
$$
since $\Delta(T^*U) \cap \pi^{-1}(\widetilde \partial_\pm \times \widetilde \partial_\mp) = \emptyset.$ By injectivity of $\dd Q_\pm^{-1}(z)^\top : T^*_z \widetilde \partial_\pm \to T^*_z\pi^{-1}(\partial D_0)$ we obtain
$$
A \subset \dd (Q_\mp^{-1})^\top \dd (\iota_\mp)^\top E_\mp^* \quad \text{and} \quad B_1 \subset \dd (Q_\pm^{-1})^\top \dd (\iota_\pm)^\top E_\pm^*.
$$
We claim that this implies $A \cap B_1 = \emptyset$. Indeed, let $(z, \xi) \in T^*Z_\pm$ which lies in $\left(\dd (Q_\mp^{-1})^\top \dd (\iota_\mp)^\top E_\mp^*\right) \cap \left(\dd (Q_\pm^{-1})^\top \dd (\iota_\pm)^\top E_\pm^*\right).$ Thus $z$ lies in $\Lambda'$ and there exists $(z_\pm, \xi_\pm)$ in $E_\pm^*$ such that 
$$
(z, \xi) = \dd(Q_\pm^{-1})^\top \dd(\iota_\pm)^\top (z_\pm, \xi_\pm).
$$
There are neighborhoods $U_\pm$ of $z_\pm$ in $M'$ and smooth functions $s_\pm : U_\pm \to \R$ such that $Q_\pm(z_\pm') = \varphi_{s_\pm(z_\pm')}(z_\pm')$ for $z_\pm' \in U_\pm$ and $\varphi_{s_\pm(z_\pm')}(z_\pm') \in Z$ for $z'_\pm \in U_\pm$. Because $\xi_\pm \in \ker X(z_\pm)$ we see that
$$
\dd (Q^{-1}_\pm)^\top \dd (\iota_\pm)^\top (z_\pm, \xi_\pm) = \dd \iota^\top \dd_{z} \left(\varphi_{-s_\pm(z_\pm)}\right)^\top \xi_\pm
$$
where $\iota : Z \hookrightarrow M'$ is the inclusion. Because $\dd \iota^\top : \ker X \to T^*Z$ is injective, we obtain
$$
\xi_- = \dd \Bigl[u \mapsto \varphi_{s_-(z_-) - s_+(z_+)}(u)\Bigr](z_-)^\top \xi_+.
$$
Now we have $z_\pm \in \Lambda'$ and since $\xi_\pm \in E_\pm^*$ we obtain $\xi_+ \in E_u'^*(z_+)$ and $\xi_- \in E_s'^*(z_-).$ Thus $\xi \in E_u'^*(z) \cap E_s'^*(z) = \{0\}$. Here we denoted by 
$$T_zM' = \R E_s'(z') \oplus \R E_u'(z') \oplus \R X(z'), \quad z' \in \Lambda',$$
the hyperbolic decomposition of $TM'$ over $\Lambda'$. We conclude that $A \cap B_1 = \emptyset$, which concludes the case $n = 2$.

By \cite[Theorem 8.2.14]{hor1} we also have the bound
$$
\WF((\varrho \Tc_\pm(s))^2) \subset \left(\WF'(\varrho \Tc_\pm(s)) \circ \WF'(\varrho \Tc_\pm(s))\right) \cup (B_1 \times \underline 0) \cup (\underline 0 \times A),
$$
where $\underline 0 \subset T^*M'$ denotes the zero section. Therefore, the set $B_2$, which is defined by
$$
B_{2}=\left\{(z, \xi) \in T^*Z~:~\exists z' \in Z,~\left(z, \xi, z^{\prime}, 0\right) \in \mathrm{WF}\left(\left(\varrho {\mathcal{T}}_{\pm}(s)^{2}\right)\right)\right\},
$$
can be written
$$
\begin{aligned}
\bigl\{ (z, \xi) \in T^*Z~:~ \exists z',z'' \in Z,~&\exists \eta \in T_{z'}^*Z, ~(z, \xi, z'-\eta) \in \WF(\varrho \Tc_\pm(s)) \\
&\quad \quad \quad  \quad \quad \quad \text{ and }(z', \eta, z'', 0) \in \WF(\varrho \Tc_\pm(s))\bigr\} \cup B_1.
\end{aligned}
$$
As $\dd (Q_+^{-1})^\top \dd (\iota_+)^\top E_+^* \cap \dd (Q_-^{-1})^\top \dd( \iota_-)^\top E_-^* = 0$ (as shown above), we obtain
$$
\begin{aligned}
B_2 \subset \Bigl\{ (z, \xi)~:~&(z, \xi, z', \eta) \in \dd(Q_\pm^{-1} \times Q_\mp^{-1})^\top \dd (\iota_\pm \times \iota_\mp)^\top \Upsilon_\pm \\
& \qquad \qquad \qquad  \text{ for some } \eta \in \dd (Q_\pm^{-1})^\top \dd (\iota_\pm)^\top(E_\pm^*) \Bigr\}.
\end{aligned}
$$
This leads to
$$
\begin{aligned}
B_2 &\subset \Bigl\{\dd (Q_\pm^{-1})^\top \dd (\iota_\pm)^\top \Phi_t(z, \zeta)~:~ (z, \zeta) \in T^*Y_\mp,~ \langle X(z), \zeta \rangle = 0,\\
&\qquad \qquad \qquad  \dd(\iota_\mp)^\top(z, \zeta) \in \dd (Q_\pm \circ Q_\mp^{-1})^\top \dd( \iota_\pm)^\top E_\pm^*,~\varphi_t(z) \in \widetilde \partial_\pm U,~t \geqslant 0 \Bigr\}.
\end{aligned}
$$
As before, this set cannot intersect $\dd (Q_\mp^{-1})^\top E_\mp^*$ since otherwise we would have $z' \in \Lambda'$ and $\xi' \in T^*_z M'$ contracted in the past and in the future by $\dd \varphi_t'^\top$. Thus $B_2 \cap A = \emptyset$ and we obtain that $(\varrho \Tc_\pm(s))^3$ is well defined. By iterating this process we obtain that $(\varrho \Tc_\pm(s))^n$ is well defined for every $n \geqslant 2$, which concludes the proof.
\end{proof}

\subsection{The flat trace of $\Tc_\pm(s)$}\label{subsec:flattrace}
Let $\mathcal{A} : \Omega^\bullet_c(\partial) \to \mathcal{D}'^\bullet(Z)$ be an operator such that
$
\WF'(\mathcal{A}) \cap \Delta = \emptyset,
$
where $\Delta$ is the diagonal in $T^*(Z \times Z)$.
Then the flat trace of $\mathcal{A}$ is defined as
$$
\strf \mathcal{A} = \langle \iota^*_\Delta A, 1 \rangle,
$$
where $\iota_\Delta : z \mapsto (z,z)$ is the diagonal inclusion and $A \in \mathcal{D}'^{n}(Z\times Z)$ is the Schwartz kernel of $\mathcal{A}$, i.e.
$$
\int_Z \mathcal{A}(u) \wedge v = \int_{Z \times Z} A \wedge \pi_1^*u \wedge \pi_2^*v, \quad u,v \in \Omega^\bullet_c(Z),
$$
where $\pi_j : Z \times Z \to Z$ is the projection on the $j$-th factor ($j=1,2$). In fact we have
\begin{equation}\label{eq:alternatedtrace}
\strf (A) = \sum_{k=0}^2 (-1)^{k+1} \mathrm{tr}^\flat(A_k),
\end{equation}
where $\mathrm{tr}^\flat$ is the transversal trace of Attiyah-Bott \cite{atiyah1967lefschetz} and where we denoted by $A_k$ the operator $C^\infty_c\bigl(Z, \wedge^kT^*Z\bigr) \to \mathcal{D}'\bigl(Z, \wedge^k T^*Z\bigr)$ induced by $A$ on the space of $k$-forms. The purpose of this subsection is to prove the following result.

\begin{prop}
For $n \geqslant 1$, the flat trace of $(\varrho \Tc_\pm(s))^n$ is well defined and we have
\begin{equation}\label{eq:trace}
\strf\left((\varrho \Tc_\pm(s))^n\right) = \sum_{r(\gamma) = n} \frac{\tau^\sharp(\gamma)}{\tau(\gamma)} \e^{-s\tau(\gamma)} \left(\prod_{z \in R(\gamma)} \varrho^2(z)\right)^{\tau(\gamma) / \tau^\sharp(\gamma)}
\end{equation}
whenever $\Re(s) \gg 1$, where the sum runs over all periodic trajectories $\gamma$ rebounding $n$ times on $\partial D_0$. Here $\tau^\sharp(\gamma)$ is the primitive length of $\gamma.$ and
$$
R(\gamma) = \{(\gamma(\tau),\dot \gamma(\tau))~:~\tau \in \R\} \cap \pi^{-1}(\partial D_0)
$$
is the set of incidence vectors of $\gamma$ along $D_0.$
\end{prop}

\begin{corr}
As $s \mapsto (\varrho \Tc_\pm(s))^n$ extends meromorphically to the whole complex plane, so does the right hand side of (\ref{eq:trace}).
\end{corr}

\begin{proof}
For $z \in Z$ we define the first (future and past) return times to $\pi^{-1}(\partial D_0)$ by 
$$
t_{\pm, 0}(z) = \inf \{t > 0~:~\varphi'_{\pm t}(z) \in \pi^{-1}(\partial D_0)\}.
$$
We set 
$
\Lambda_{\pm, 0} = \{z \in Z~:~t_{\pm, 0}(z) < \infty\},
$
and we define by $B_{\pm,0} : Z \to \pi^{-1}(\partial D_0)$ the first (future and past) return maps to $\pi^{-1}(\partial D_0)$ by 
$$
B_{\pm, 0}(z) = \varphi'_{t_{\pm, 0}}(z), \quad z \in \Lambda_{\pm, 0}.
$$
For $n \geqslant 1$ we define the sets $\Lambda^{(n)}_{\pm, 0} \subset Z$ by induction as follows. We set 
$
\Lambda^{(1)}_{\pm, 0} = \Lambda_{\pm, 0}
$
and 
$$
\Lambda^{(n + 1)}_{\pm, 0} = \left\{z \in \Lambda_{\pm, 0}~:~B_{\pm, 0}(z) \in \Lambda_{\pm,0}^{(n)}\right\}, \quad n \geqslant 1.
$$
In particular $(B_{\pm, 0})^n(z)$ is well defined for $z \in \Lambda_{\pm, 0}^{(n)}.$ We finally set
$$
t_{\pm,0}^{(n)}(z) = \sum_{k=0}^{n-1} t_{\pm, 0}\left((B_{\pm, 0})^k(z)\right), \quad z \in \Lambda_{\pm, 0}^{(n)},
$$
and $t_{\pm, 0}^{(n)}(z) = +\infty$ for $z \in Z \setminus \Lambda^{(n)}_{\pm, 0}$. We now fix $n \geqslant 1.$ Let $g \in C^\infty(\R, [0,1])$ such that $g \equiv 1$ on $]-\infty, 1]$ and $g \equiv 0$ on $[2, +\infty[$. For $L > 0$ we define
$$
g_L(z) = g\left(t_{\pm, 0}^{(n)}(z) - L\right), \quad z \in Z.
$$
Then by definition of $\Tc_\pm(s)$, the operator $g_L \left(\varrho \Tc_\pm(s)\right)^n : \Omega^\bullet_c(Z) \to \mathcal{D}'^\bullet(Z)$ coincides with the operator
$$
w~\longmapsto~ g_L(\cdot) \left(\prod_{k = 0}^n \varrho^2\left((B_{\mp, 0})^k(\cdot)\right) \right) \e^{-s t_{\pm, 0}^{(n)}(\cdot)} \left((B_{\mp, 0})^n\right)^*w.
$$
It now follows from the Atiyah-Bott trace formula \cite[Corollary 5.4]{atiyah1967lefschetz} that\footnote{See the proof of \cite[Proposition 3.6]{chaubet2021closed} for more details.}
\begin{equation}\label{eq:pretrace}
\langle \iota_\Delta^*K_{\varrho, \pm}(s), g_L\rangle = \sum_{\substack{z \in Z \\ B_{\mp, 0}^n(z) = z}} \e^{-t_{\pm, 0}^{(n)}(z)} g_L(z) \left(\prod_{k=0}^{n-1} \varrho^2\left(B_{\mp}^k(z)\right)\right).
\end{equation}
It is a classical fact that for every $k,n\geqslant 1$, there is $C_k > 0$ such that 
$$
\|\dd^k\left((B_{\pm, 0})^n\right)(z)\| \leqslant C_k \exp\left(C_k t_{\pm, 0}^{(n)}(z)\right), \quad z \in \Lambda_{\pm,0}^{(n)}.
$$
Thus we may proceed exactly as in the proof of \cite[Proposition 3.6]{chaubet2021closed} to take the limit in (\ref{eq:pretrace}) when $L \to +\infty$ to obtain (\ref{eq:trace}).
\end{proof}

\section{A Tauberian argument}\label{sec:tauberian}
In this section we use a Tauberian theorem of Delange \cite{delange1954generalisation} to derive an asymptotic growth of a weighted sum of periodic trajectories rebounding a fixed number of times on $\partial D_0$.

\subsection{Zeta functions}
Let $\Pc_\bold{B}$ be the set of primitive periodic orbits of $(\varphi_t)$, for the billiard table $\bold{B}$. We define the Ruelle zeta function $\zeta_{\bold{B}}$ associated to the billiard flow $\bold{B}$ by
$$
\zeta_\bold{B}(s) = \prod_{\gamma \in \Pc_\bold{B}}\left(1 - \e^{-s \tau(\gamma)}\right)^{-1}, \quad s \in \C,
$$
where the product converges whenever $\Re(s)$ is large enough. By \cite[Theorem 1.3]{morita2007meromorphic}, there is $h_\bold{B} > 0$ and $c_\bold{B} > 0$ such that $\zeta_{\bold{B}}$ admits a holomorphic extension on $\{\Re(s) > -c_\bold{B}\} \setminus \{h_\bold{B}\}$, and has a simple pole with residue $1$ at $s = h_\bold{B}$\footnote{In fact, $\zeta_\bold{B}$ admits a meromorphic continuation to the whole complex plane by \cite{kuster2021smooth}.}. On the other hand, it follows from \cite{dyatlov2016pollicott} (see \cite[\S6]{kuster2021smooth}) that we may write, for $\Re(s)$ large enough, any $\varepsilon > 0$ small and $\chi \in C^\infty_c(U, [0,1])$ satisfying $\chi \equiv 1$ on $\Lambda$,
\begin{equation}\label{eq:zetaresolv}
\partial_s \log \zeta_\bold{B}(s) = \sum_{k=0}^2 (-1)^k \e^{\mp \varepsilon s} \mathrm{tr}^\flat\left(\chi \widetilde \varphi_{\mp \varepsilon}^* \widetilde R_\pm(s) \chi|_{\Omega^k_0}\right).
\end{equation}
Here we set $\Omega^k_0 = \{w \in \Omega^k(U)~:~\iota_{\widetilde X} w = 0\}.$ This implies that
$
\mathrm{rank}(\Pi_\pm(h_\bold{B})) = 1
$
(with the notations of \S\ref{subsec:resolvent}). Moreover, we know that $s \mapsto \widetilde R_\pm(s)|_{\Omega^0_0}$ is holomorphic on $\{\Re(s) > 0\}$ (since the integral formula converges absolutely in this region). We claim that
$$
\Pi_\pm(s)|_{\Omega^2_0} = 0, \quad \Re(s) > 0.
$$
Indeed, since $X = \widetilde X$ near $\Gamma_+ \cup \Gamma_-$, we have for any open neighborhood $W$ of $\Gamma_+ \cup \Gamma_-$ in $U$ not intersecting $\supp(X - \widetilde X)$
$$
\Omega^2_0(W) = \Omega^0_0(W) \wedge \dd \alpha.
$$
By \S\ref{subsec:resolvent} we have $\supp(\Pi_\pm(s)) \subset \Gamma_\pm \times \Gamma_\mp.$ As $\Lie_X$ and $w \mapsto w \wedge \dd \alpha$ commute near $W$ we thus obtain, with the above decomposition
$$
\Pi_\pm(s)|_{\Omega^2_0}w = (\Pi_\pm(s)|_{\Omega^0_0} v) \wedge \dd \alpha, \quad w = v \wedge \dd \alpha \in \Omega^2_0(W).
$$
In particular, for any $s$ with $\Re(s) > 0$, we have $\Pi_\pm(s)|_{\Omega^2_0} = 0$ since $\Pi_\pm(s)|_{\Omega^0_0} = 0$ (as $s \mapsto \widetilde R_\pm(s)|_{\Omega^0_0}$ is holomorphic in $\{\Re(s) > 0\}$).  Thus, we obtained that the only pole of $s \mapsto \widetilde R_\pm(s)|_{\Omega_0^\bullet}$ in $\{\Re(s) > 0\}$ is located at $s = h_\bold{B}$. The associated residue is 
$
\Pi_\pm(h_\bold{B})|_{\Omega^1_0},
$
which is of rank one by (\ref{eq:zetaresolv}) as $\zeta_\bold{B}$ has a simple pole with residue one at $s = h_\bold{B}$.

\subsection{A Tauberian argument}\label{subsec:tauberian}

Taking the notations of paragraphs \ref{subsec:scatteringop}, \ref{subsec:notations} and \ref{subsec:composingscat}, we set
$$
A_\pm = - \Qc_\pm(h_\bold{B})\chi_\pm \iota_\pm^* \iota_{\widetilde X} \Pi_\pm(h_\bold{B}) (\iota_\mp)_* \chi_\mp \Qc_\mp(h_\bold{B})^\top.
$$
Then by Proposition \ref{prop:scatresolv}, we have, as operators $\Omega^\bullet_c(Z) \to \mathcal{D}'^\bullet(Z),$
$$
(\varrho \Tc_\pm(s))^n = \frac{(A_\pm)^n}{(s-h_\bold{B})^n} + \dom((s-h_\bold{B})^{-n + 1}), \quad s \to h_\bold{B}.
$$
As $A_\pm$ is of rank one, we have $\strf((A_\pm)^n) = (\strf(A_\pm))^n$. Thus letting $c_\pm = \strf(A_\pm)$ we have
\begin{equation}\label{eq:poletrace}
\strf\left((\varrho \Tc_\pm(s))^n\right) = \frac{(c_\pm)^n}{(s - h_\bold{B})^n} + \dom((s-h_\bold{B})^{-n + 1}), \quad s \to h_\bold{B}.
\end{equation}
Now we define
$$
N_\varrho(t, n) = \sum_{\substack{\gamma \in \Pc \\ r(\gamma) = n \\ \tau(\gamma) \leqslant t}} I_\varrho(\gamma), \quad t \geqslant 0,
$$
where we set, for a closed trajectory $\gamma:[0,\tau(\gamma)] \to M'$,
$$
I_\varrho(\gamma) = \prod_{z \in R(\gamma)} \rho^2(\gamma) \quad \text{where} \quad R(\gamma) = \pi^{-1}(D_0) \cap \{(\gamma(\tau), \dot \gamma(\tau))~:~\tau \in [0,\tau(\gamma)]\}.
$$
Note that if $r(\gamma) = n$ one has $\sharp R(\gamma) = n \tau(\gamma) / \tau^\sharp(\gamma)$.

\begin{prop}\label{prop:firstequivalent}
Assume that $c_\pm > 0.$ Then
$$
N_\varrho(t, n) \sim \frac{(c_\pm t)^n}{n!} \frac{\e^{h_\bold{B} t}}{h_\bold{B}t}, \quad t \to +\infty.
$$
\end{prop}

\begin{proof}
Here we follow the proof of \cite[Lemma 5.1]{chaubet2021closed}. Define
$$g_{n, \varrho}(t) = \sum_{\substack{\gamma \in \Pc \\ r(\gamma) = n}} \tau(\gamma) \sum_{\substack{k \geqslant 1 \\ k \tau(\gamma) \leqslant t}} I_\varrho(\gamma)^k, \quad t \geqslant 0.$$
For $\Re(s)$ large enough we set $\displaystyle{G_{n, \varrho}(s) = \int_0^\infty g_{n, \varrho}(t) \e^{-ts} \dd t}$. Then a simple computation shows that
$$
G_{n, \varrho}(s) = \frac{1}{s} \sum_{r(\gamma) = n} \tau^\sharp(\gamma) I_\varrho(\gamma)^{\tau(\gamma) / \tau^\sharp(\gamma)} \e^{-s \tau(\gamma)} = - \frac{\partial_s \strf\left((\varrho \Tc_\pm(s))^n\right)}{ns},
$$
where the sum runs over all periodic orbits (not necessarily primitive) $\gamma$ such that $r(\gamma) = n$.
By (\ref{eq:poletrace}) we have
$$
G_{n, \varrho}(h_\bold{B}s) = \frac{(c_\pm)^n}{h_\bold{B}^{n+2} (s - 1)^{n+1}} + \dom((s-1)^{-n}), \quad s \to 1.
$$
Then applying a Tauberian theorem from Delange \cite[Th\'eor\`eme III]{delange1954generalisation} we obtain
$$
\frac{1}{h_\bold{B}} g_{n, \varrho}(t / h_\bold{B}) \sim \frac{(c_{\pm})^{n}}{h_{\bold{B}}^{n+2}} \frac{\mathrm{e}^{t}}{n !} t^{n}, \quad t \rightarrow+\infty,
$$
which reads
$
\displaystyle{
g_{n, \chi}(t) \sim \frac{\left(c_{\pm} t\right)^{n}}{n ! h_{\bold{B}}} \exp \left(h_{\bold{B}} t\right)
}
$
as $t \to +\infty.$ Now note that
$$
g_{n, \varrho}(t) \leqslant \sum_{\substack{\gamma \in \Pc \\ r(\gamma) = n \\ \tau(\gamma) \leqslant t}} \tau(\gamma) \lfloor t / \tau(\gamma) \rfloor I_\varrho(\gamma) \leqslant t N_\varrho(t)
$$
which gives $\displaystyle{\liminf_{t \to +\infty}} \frac{N_\varrho(t)}{g_{n, \varrho}(t) / t} \geqslant 1.$ On the other hand, let 
$$
\zeta_{n, \varrho}(s) = \prod_{\substack{\gamma \in \Pc \\ r(\gamma) = n}} \left(1 - I_\varrho(\gamma) \e^{-s \tau(\gamma)}\right)^{-1}, \quad \Re(s) \gg 1.
$$
Then we have
\begin{equation}\label{eq:zetamajor}
\begin{aligned}
 \zeta_{n, \varrho}(s) &\geqslant \prod_{\substack{\gamma \in \Pc \\ r(\gamma) = n}} \left(1 + I_\varrho(\gamma) \e^{-s\tau(\gamma)}\right) 
&\geqslant\prod_{\substack{\gamma \in \Pc \\ r(\gamma) = n \\ \tau(\gamma) \leqslant t}} \left(1 + I_\varrho(\gamma) \e^{-s t} \right) 
&\geqslant   \e^{-st}N_\varrho(t).
\end{aligned}
\end{equation}
As $\partial_s \log \zeta_{n, \varrho}(s) = -s G_{n \varrho}(s)$, it follows that $\zeta_{n, \varrho}$ extends holomorphically on $\{\Re(s) > h_\bold{B}\}$ (as $G_{n, \rho}$ does). Let $\sigma > 1,$ and $\varepsilon > 0$ such that $(h_\bold{B} + \varepsilon)  / \sigma < h_\bold{B}.$ Then by (\ref{eq:zetamajor}) applied with $s = h_\bold{B} + \varepsilon$ we have

$$
N_\varrho(t / \sigma) \leqslant \zeta_{n, \varrho}(h_\bold{B} + \varepsilon) \exp\left(\frac{(h_\bold{B} + \varepsilon)t}{\sigma} \right).
$$

This implies that $N_\varrho(t/\sigma) / N_\varrho(t) \to 0$ as $t \to +\infty.$
Now we write
$$
g_{n, \varrho}(t) \geqslant \sum_{\substack{\gamma \in \Pc \\ r(\gamma) = n \\ \tau(\gamma) \leqslant t}} \tau(\gamma) I_\varrho(\gamma) \geqslant  \sum_{\substack{\gamma \in \Pc \\ r(\gamma) = n \\ t / \sigma \leqslant \tau(\gamma) \leqslant t}} \frac{t}{\sigma} I_\varrho(\gamma) = \frac{t}{\sigma} \left(N_\varrho(t) - N_\varrho(t/\sigma)\right).
$$
This leads to
$$
\limsup_{t \to +\infty} \frac{N_\varrho(t)}{g_{n, \varrho}(t) / t} \leqslant \sigma \limsup_{t \to +\infty}
\left(1 - \frac{N_\varrho(t/ \sigma)}{N_\varrho(t)}\right)^{-1} = \sigma.$$
As $\sigma > 1$ is arbitrary, the proof of the lemma is complete, since we have
$$
g_{n, \varrho}(t)/t \sim \frac{(c_\pm t)^n}{n!} \frac{\e^{h_\bold{B} t}}{h_\bold{B}t}
$$
as $t$ goes to infinity.
\end{proof}

\section{A priori bounds}\label{sec:apriori}
In this section we derive some \textit{a priori} bounds on $N(n, t)$ (the number of primitive periodic orbits bouncing $n$ times on $\partial D_0$ and of length not greater than $t$) by using the fact that the billiard flow is conjugated to a subshift of finite type. This will allow us to convert the asymptotics obtained in \S\ref{sec:tauberian} into an asymptotics on $N(n,t).$

\subsection{Coding}

Let $\Sigma_N'$ be the set of finite sequences $u = u_1 \cdots u_{N}$ with $u_j \in \{0, 1, \dots, r\}$ and $u_{j} \neq u_{j + 1}$ (with $j \in \Z/N\Z$), and such that $u$ is distinct from its cyclic permutations. We also define $\Sigma_{N}$ as above by replacing $\{0, 1, \dots, r\}$ by $\{1, \dots, r\}$.  By \S\ref{subsec:anosov} we have a one-to-one correspondance
\begin{equation}\label{eq:correspondance}
\Pc_\bold{B'} \longleftrightarrow \left(\bigcup_{N = 2}^\infty \Sigma_N' \right) / \sim
\end{equation}
where $u \sim v$ if and only if $u$ is a cyclic permutation of $v$. For any $\gamma \in \Pc_\bold{B'}$ we will denote by $\wl(\gamma)$ its word length, that is, the length of (any) word which is associated to $\gamma$ via the above correspondance.

For any  sequence $u \in \Sigma_N'$, we will denote by $\gamma_u : \R \to \Lambda'$ the closed billiard trajectory (parametrized by arc length) starting from the point $z_u \in K'$ which is associated to the sequence 
$$
(\cdots u u u \cdots) \in \Sigma'.
$$
Its period is then defined by 
$$
\tau(\gamma_u) = \sum_{k=0}^{N-1} t_+'(B'^k(z_u)),
$$
where $t_+$ is defined in \S\ref{subsec:anosov}. We have the following result.

\begin{lemm}\label{lem:comparelength}
There is $C > 0$ such that the following holds.
Let $\gamma : [0, T] \to \Lambda'$ be a billiard trajectory (parametrized by arc length) such that $\gamma(0), \gamma(T) \in \pi^{-1}(\partial D_0)$ and denote by $0 = t_0 < \cdots < t_N = T$ the times for which $\gamma$ hits $\partial D$ and assume that $N > 2$. Let $u = u_1 \cdots u_{N-1} \in \{0, \dots, r\}^{N - 1}$ be the finite sequence such that it holds
$
\pi(\gamma(t_k)) \in \partial D_{u_k}
$
for
$
k = 1, \dots, N-1,
$
and assume that $u_1 \neq u_{N-1}$ so that $\gamma_u$ is well defined.
Then 
$$
\tau(\gamma_u) - C \leqslant T \leqslant \tau(\gamma_u) + C.
$$
\end{lemm}

\begin{proof}
By Lemma \ref{lem:exponentiallyclose}, it holds, for some $C > 0$ and $\beta > 1$ which are independent of $\gamma$,
$$
\mathrm{dist}(B'^k(z_u), \gamma(t_k)) \leqslant C \beta^{-N / 2 + |k - N/2|}, \quad k = 1, \dots, N-1.
$$
Now note that $t_+' : \{z \in \pi^{-1}(\partial D)~:~t_+'(z) < +\infty\} \to \R_+$ is locally Lipschitz continuous. As $K'$ is compact, it follows that for some $C' > 0$ we have
$$
\left|t_+'(B'^k(z_u)) - t_+'(\gamma(t_k)) \right| \leqslant C' \beta^{-N / 2 + |k - N/2|}
$$
and thus
$$
\left| \tau(\gamma_u) - T \right| \leqslant 2 L_m + C' \sum_{k=1}^{N-1} \beta^{-N / 2 + |k - N/2|} \leqslant 2L_m + \frac{C'}{\beta - 1},
$$
where 
$L_m = \sup \{\mathrm{dist}(x_i, x_j)~:~x_i \in D_i,~x_j \in D_j,~i \neq j\}.$
This concludes the proof.
\end{proof}

\subsection{The bounds}
Let $\Pc_\bold{B}$ be the set of oriented primitive periodic orbits of the flow associated to
the billiard $\bold{B}$, and set $\Pc_\bold{B}(t) = \{\gamma \in \Pc_\bold{B}~:~\tau(\gamma) \leqslant t\}. $ Then by \cite{morita1991symbolic} we have
\begin{equation}\label{eq:margulis}
\sharp\{\gamma \in \Pc_\bold{B}~:~\tau(\gamma) \leqslant t\} \sim \frac{\e^{h_\bold{B} t}}{h_\bold{B}t}, \quad t \to +\infty.
\end{equation}
In what follows, we will denote by $\Pc_\bold{B'}(n, t)$ the set of primitive periodic trajectories of the billiard $\bold{B}'$ of period less than $t$ which make exactly $n$ rebounds on $\partial D_0$, and $N(n, t) = \sharp\Pc_{\bold{B}'}(n,t)$. Finally we denote by $ \widetilde \Pc_\bold{B}(t)$ (resp. $\widetilde \Pc_\bold{B'}(n,t)$) the set of (not necessarily primitive) periodic orbits for the billiard $\bold{B}$ (resp. for the billiard $\bold{B'}$) of period less or equal than $t$ (resp. and making $n$ rebounds on $\partial D_0$) ; we denote $\widetilde N(t) = \sharp \widetilde \Pc_\bold{B}(t)$ and $\widetilde N(n,t) = \sharp \widetilde \Pc_\bold{B'}(n,t)$. It is a classical fact that we have
\begin{equation}\label{eq:primnotprim}
\widetilde{N}(t) \sim N(t), \quad t \to +\infty,
\end{equation}
as it can be seen from the equalities
$$
\widetilde N(t) = \sum_{\tau(\gamma) \leqslant t} 1 = \sum_{\substack{\gamma \in \Pc}} \sum_{k \tau(\gamma) \leqslant t} 1= \sum_{\substack{\gamma \in \Pc \\ t/2 < \tau(\gamma) \leqslant t}} 1 + \sum_{\substack{\gamma \in \Pc \\ \tau(\gamma) \leqslant t / 2}} \lfloor t/\tau(\gamma) \rfloor,
$$
and the fact that $\sum_{\substack{\gamma \in \Pc \\ \tau(\gamma) \leqslant t / 2}} \lfloor t/\tau(\gamma) \rfloor \ll N(t)$ as $t \to +\infty$ by (\ref{eq:margulis}).

\begin{prop}\label{prop:aprioribound}
For each $n \geqslant 1$, there is $C_n>0$ such that if $t$ is large enough we have
\begin{equation}\label{eq:aprioribound}
C_n^{-1} t^{n-1} \exp(h_\bold{B}t) \leqslant N(n, t) \leqslant C_n t^{n-1} \exp(h_\bold{B} t).
\end{equation}
\end{prop}

\begin{proof}
We start with the case $n = 1$. Consider the map $F : \Sigma_N \to \Sigma_{N+1}'$ defined by $F(u_1 \cdots u_N) = 0u_1 \cdots u_N$ (note that for any word $u \in \Sigma_N$, $F(u)$ is still a primitive word as it contains exactly one zero in its letters). By Lemma \ref{lem:comparelength}, we have
$$
\tau(\gamma_u) - C \leqslant \tau(\gamma_{F(u)}) \leq \tau(\gamma_u) + C, \quad u \in \Sigma_N.
$$
The map $F$ is obviously injective. Recalling the correspondance (\ref{eq:correspondance}) (for both billiards $\bold{B}$ and $\bold{B'}$), we thus have
$$
\begin{aligned}
N(1, t) &\geqslant \sum_{N = 2}^\infty \sum_{\substack{u \in \Sigma_N \\ \tau(\gamma_u) \leqslant t - C}} 1 &= \sum_{\gamma \in \Pc_\bold{B}(t-C)} \mathrm{wl}(\gamma),
\end{aligned}
$$
where the last equality comes from the fact that each $\gamma \in \Pc_\bold{B}$ corresponds to exactly $\wl(\gamma)$ words in $\Sigma.$
Note that for some $C > 0$ it holds
\begin{equation}\label{eq:comparewordlength}
C^{-1}\tau(\gamma) \leqslant \wl(\gamma) \leqslant C\tau(\gamma), \quad \gamma \in \Pc_\bold{B}.
\end{equation}
In particular we obtain
$$
N(1, t) \geqslant \frac{t}{2C} \sharp\left(\Pc_\bold{B}(t) \setminus \Pc_\bold{B}(t/2)\right).
$$
By (\ref{eq:margulis}), we obtain that the first inequality of (\ref{eq:aprioribound}) holds for $n = 1$. For the second one, consider the set $\widetilde \Sigma_N$ of finite words $u_1 \cdots u_N$ with $u_j \neq u_{j+1}$ for $j \in \Z/N\Z$ (note that $\Sigma_N \subset \widetilde \Sigma_N$ is the set of primitive words within $\widetilde \Sigma_N$). Consider the map $G : \widetilde \Sigma_N  \to \Sigma_{N+2}'$ defined by
$$
G(u_1 \cdots u_N) = 0 u_1 \cdots u_N u_1, \quad u_1 \cdots u_N \in \widetilde \Sigma_N.
$$
Every primitive periodic orbit bouncing exactly one time on $\partial D_0$ can be encoded by a finite word of the form $F(u)$ or $G(u)$ for some $u \in \widetilde \Sigma_N$ where $N \geqslant 2$ (note that $F$ extends to a map $F : \widetilde \Sigma_N \to \Sigma_{N+1}'$).
In particular, by Lemma \ref{lem:comparelength}, we have for some $C > 0$
$$
\Pc(1, t) \subset \bigcup_{N} \left[F\left(\left\{u \in \widetilde \Sigma_N~:~\tau(\gamma_u) \leqslant t + C \right\}\right) \cup G\left(\left\{u \in \widetilde \Sigma_N~:~\tau(\gamma_u) \leqslant t + C\right \}\right)\right].
$$
With (\ref{eq:comparewordlength}) in mind, this leads to
$$
N(1,t) \leqslant 2 \sum_{N=2}^\infty \sum_{\substack{u \in \widetilde \Sigma_N \\ \tau(\gamma_u) \leqslant t + C}} 1 \leqslant 2 \sum_{\substack{\gamma \in \widetilde \Pc_{\bold{B}} \\\tau(\gamma) \leqslant t + C}} \wl(\gamma) \leqslant 2 (t+C) \widetilde N(t + C) \leqslant C \exp(h_\bold{B}t),
$$
where the last inequality holds for $t$ large enough and comes from (\ref{eq:primnotprim}). The case $n = 1$ is proven.

We now proceed by induction and assume that (\ref{eq:aprioribound}) holds for every $n = 1, \dots, m$, for some $m \geqslant 1.$ Similarly to (\ref{eq:primnotprim}), the estimate (\ref{eq:aprioribound}) also holds if we replace $N(n,t)$ by $\widetilde N(n,t)$. Every element of $\Pc_\bold{B'}(m + 1, t)$ can be represented by the concatenation of a word (starting from $0$) representing an element of $\widetilde \Pc_{\bold{B}'}(m,t_1)$ and a word (starting from $0$) representing an element of $\widetilde \Pc_{\bold{B}'}(1,t_2)$, where $t_1 + t_2 \leqslant t + 2C$ (for some constant $C$). More precisely, for $N, k \geqslant 1$, set
$$
A(k) = \left\{u_1 \cdots u_N \in \widetilde \Sigma_N'~:~N \geqslant 2,~~u_1 = 0, ~~u_N \neq 0,~~\sharp\{j : u_j = 0\} =  k\right\}.
$$
Then every element $\gamma$ of $\Pc_\bold{B'}(m+1,t)$ can be represented by a word $uv$ (i.e. $\gamma = \gamma_{uv}$) where $u \in A(m)$ and $v \in A(1)$. Moreover, by Lemma \ref{lem:comparelength}, we must have
\begin{equation}\label{eq:comparelengthuv}
\tau(\gamma) - 2C \leqslant \tau(\gamma_u) + \tau(\gamma_v) \leqslant \tau(\gamma) + 2C
\end{equation}
for some $C$ which does not depend of $\gamma.$ Note also that for each periodic trajectory making $k$ rebounds on $\partial D_0$, there are at most $k$ words in $A(k)$ representing it (since the words have to start by the letter $0$).  Summarizing the above facts, we have for $t$ large enough (in what follows $C$ is a constant depending only on $m$ that may change at each line)
$$
\begin{aligned}
\widetilde N(m+1, t) &\leqslant \sum_{\substack{u \in A(m) \\ \tau(\gamma_u) \leqslant t + C}} \sum_{\substack{v \in A(1) \\ \tau(\gamma_v) \leqslant t - \tau(\gamma_u) + C}} 1 \\
& \leqslant \sum_{\substack{u \in A(m) \\ \tau(\gamma_u) \leqslant t + C}} \widetilde N(1, t - \tau(\gamma_u) + C) \\
& \leqslant \sum_{\substack{u \in A(m) \\ \tau(\gamma_u) \leqslant t + C}} C \exp(h_\bold{B}(t-\tau(\gamma_u) + C)) \\
&  \leqslant  \sum_{k=1}^{t + C} m\widetilde N(m,k) C \exp(h_\bold{B}(t - k + C)) \\
& \leqslant C \sum_{k=1}^{t + C} k^{m-1}\exp(h_\bold{B}k) \exp(h_\bold{B}(t - k + C))  \\
& \leqslant C t^{m} \exp(h_\bold{B}t),
\end{aligned}
$$
where we used $\widetilde N(m,t) \leqslant Ct^{m-1}\exp(h_\bold{B})$ as it follows from the induction hypothesis. For the lower bound, we proceed as follows. The map $A(m) \times A(1) \to A(m+1)$ defined by $(u,v) \mapsto uv$ is injective ; moreover, every element of $\widetilde \Pc_\bold{B'}(m+1, t)$ is represented by exactly $m+1$ elements of $A(m+1)$. By (\ref{eq:comparelengthuv}), we have
$$
\widetilde N(m+1,t)  \geqslant \frac{1}{m+1} \sum_{\substack{u \in A(m) \\ \tau(\gamma_u) \leqslant t - C}} \sum_{\substack{v \in A(1) \\ \tau(\gamma_v) \leqslant t - \tau(\gamma_u) - C}} 1.
$$
Let $T > 0$ large enough (it will be chosen later). By similar computations as above, we have
\begin{equation}\label{eq:almostend}
\begin{aligned}
\widetilde N(m+1,t) \geqslant C \sum_{k=1}^{(t - C) / T} \left(\widetilde N(m, (k + 1) T) - \widetilde N(m, kT)\right) \exp(h_\bold{B}(t-(k+1)T-C)).
\end{aligned}
\end{equation}
If $k$ is large enough, we have by the induction hypothesis
$$
\begin{aligned}
\widetilde N(m, (k+1)T) - \widetilde N(m,kT) &\geqslant C_m^{-1} [(k+1)T]^{m-1} \e^{h_\bold{B}(k+1)T} - C_m [kT]^{m-1}\e^{h_\bold{B}kT} \\
&\geqslant (kT)^{m-1} \e^{h_\bold{B}kT}\left(C_m^{-1}\left(1 + \frac{1}{k}\right)^{m-1}\e^{h_\bold{B}T} - C_m\right).
\end{aligned}
$$
If $T$ is large enough the last term of the above equation is bounded from below by $C (kT)^{m-1} \e^{h_\bold{B}kT}$ for some $C > 0$ independent of $k$. Injecting this in (\ref{eq:almostend}), we obtain
$$
\begin{aligned}
\widetilde N(m+1,t) &\geqslant C \sum_{k=1}^{(t-C)/T} (kT)^{m-1} \exp(h_\bold{B}kT) \exp(h_\bold{B}(t-(k+1)T-C)) \\
&\geqslant C t^{m} \exp(h_\bold{B}t).
\end{aligned}
$$
Thus we proved that (\ref{eq:aprioribound}) holds for $\widetilde N(m+1,t).$ We now show that this also holds for $N(m+1, t)$, as follows. Because of Lemma \ref{lem:comparelength} and the fact that any nonprimitive word in $A(m+1)$ can be written as the concatenation of $(m+1) / d$ identical words (where $d < m + 1$ is a divisor of $m+1$) we have, for $t$ large enough,
$$
\begin{aligned}
\widetilde N(m+1, t) - N(m+1,t) &\leqslant \sum_{d~|~m + 1} \widetilde N\left(d, \frac{td}{m+1} + C\right) \\
&\leqslant C \sum_{d ~|~m + 1}\left(\frac{td}{m+1}\right)^{d-1} \exp\left(h_\bold{B} \left(\frac{td}{m+1} + C\right)\right),
\end{aligned}
$$
where the sums run over the divisors of $m + 1$ which are stricty less than $m + 1$. In particular, we have $\widetilde N(m + 1, t) - N(m+1, t) \leqslant t^{(m+1) / 2}\exp(h_\bold{B} t / 2)$ for $t$ large, and thus $N(m+1, t)$ also satisfies (\ref{eq:aprioribound}). This concludes the proof.
\end{proof}

\section{Proof of the main result}\label{sec:proof}

In this section we prove the estimate annouced in the introduction. In fact, we will prove that
$
N_\varrho(n, t) \sim N(n,t)
$
as $t \to +\infty$, which will imply the sought result.

\subsection{First considerations}

If $\gamma : \R/\tau(\gamma) \Z \to \Lambda'$ is a periodic orbit rebounding exactly $n$ times on $\partial D_0$, we denote $I_1(\gamma), \dots, I_n(\gamma) \subset \R / \tau(\gamma) \Z$ the cyclically ordered sequence of intervals satisfying $\gamma(I_j^\circ) \notin \partial D_0$ for each $j$, where $I_j^\circ$ denotes the interior of $I_j$ (this sequence is unique modulo cyclic permutations). We start by the following easy result.

\begin{lemm}\label{lem:1}
There is $t_0 > 0$ such that the following holds. For every $\gamma \in \widetilde \Pc_\bold{B'}$ such that 
$$\ell(I_j(\gamma)) \geqslant t_0, \quad j = 1, \dots, n,$$
we have $I_\varrho(\gamma) = 1.$
\end{lemm}

\begin{proof}
Let $\gamma$ as above (for some large $t_0 > 0$ which will be chosen later) and $z \in R(\gamma)$ (see \S\ref{subsec:flattrace}). Let $z_\pm = B'_\pm(z)$. Then $z_\pm \in \Lambda_{\pm}^{(m)}$, where $m = m(t_0) \to +\infty$ as $t_0 \to +\infty$. Here we set
$$
\Lambda_\pm^{(m)} = \{z \in M~:~ \sharp T_\pm(z) \geqslant m\}.
$$
In particular, by the proof of Lemma \ref{lem:comparelength} we have
$
\mathrm{dist}(z_\pm, \Gamma_\pm) \leqslant C \beta^{-m}.
$
Thus if $t_0 > 0$ is big enough, we have $\chi_\pm(z_\pm) = 1$ since $\chi_\pm \equiv 1$ on $\Gamma_\pm.$ As a consequence $\varrho(z) = 1$, also by definition of $\varrho$. Finally we have $I_\varrho(\gamma) = \prod_{z \in R(\gamma)} \varrho(z)^2 = 1$.
\end{proof}
For any $t_0 > 0$ we will denote $\widetilde N(n,t_0, t) = \sharp \widetilde \Pc_\bold{B'}(n, t_0, t)$ where
$$
\widetilde \Pc_\bold{B'}(n, t_0, t) = \{\gamma \in \widetilde \Pc_{\bold{B'}}~:~r(\gamma) = n, ~\ell(I_j(\gamma)) \leqslant t_0 \text{ for some } 1\leqslant j \leqslant n\}.
$$

\begin{lemm}\label{lem:2}
Let $t_0 > 0$ and $n \geqslant 2$. Then for some $C > 0$ we have for $t$ large enough
$$
\widetilde N (n, t_0, t) \leqslant C t^{n-2} \exp(h_\bold{B} t).
$$
\end{lemm}

\begin{proof}
By Lemma \ref{lem:comparelength}, there is $C > 0$ such that the following holds. Every trajectory $\gamma \in \widetilde \Pc_\bold{B'}(n, t_0, t)$ can be represented by a word in $\widetilde \Sigma_N'$ obtained by the concatenation of two words $u \in A(n-1)$ and $v \in A(1)$ satisfying
$$
\tau(\gamma_u) \leqslant t + C, \quad \tau(\gamma_v) \leqslant t_0 + C. 
$$
Now for $t$ large enough one has
$$
\sharp\{u \in A(n-1) ~:~ \tau(\gamma_u) \leqslant t +C\} \leqslant (n-1) (t+C)^{n-2} \exp(h_\bold{B} t)
$$
by Proposition \ref{prop:aprioribound}. As $\{v \in A(1)~:~\tau(\gamma_v) \leqslant t_0 + C\}$ is finite, the lemma is proven.
\end{proof}

\subsection{Proof of Theorem \ref{thm:main}}
First, we note that the constants $c_\pm$ given in \S\ref{subsec:tauberian} is positive. Indeed, if $c_\pm = 0$, then $s \mapsto \strf(\varrho \Tc_\pm(s))$ would be regular at $s = h_\bold{B}$ by the proof of Proposition \ref{prop:firstequivalent}. In particular, we would have
$$
N_\varrho(1, t) \ll \exp(h_\bold{B} t), \quad t \to \infty.
$$
However, by Lemma \ref{lem:1}, we have $I_\varrho(\gamma) = 1$ whenever $\tau(\gamma)$ is large enough and $r(\gamma)  =1 $, which gives $N_\varrho(1, t) \sim N(1,t)$ as $t \to \infty.$  Now $N(1,t) \geqslant C \exp(h_\bold{B} t)$ for large $t$ by Proposition \ref{prop:aprioribound}, which contradicts the fact that $N_\varrho(1, t) \ll \exp(h_\bold{B} t)$. Thus $c_\pm > 0$.
By Lemmas \ref{lem:1} and \ref{lem:2} we have
$$
N(n,t) - N_\varrho(n, t) \leqslant N(n, t, t_0) \leqslant Ct^{n-2} \exp(h_\bold{B} t).
$$
Thus, by Propositions \ref{prop:firstequivalent} and \ref{prop:aprioribound}, we obtain $N_\varrho(n, t) \sim N(n, t)$ as $t\to \infty$, which reads
$$
N(n,t) \sim \frac{(c_\pm t)^n}{n !} \frac{\e^{h_\bold{B}t}}{h_\bold{B} t}, \quad t \to \infty.
$$
This concludes the proof of Theorem \ref{thm:main}.

\bibliographystyle{alpha}
\bibliography{bib.bib}

\end{document}